\documentclass[11pt]{article}

\usepackage{amsmath,amsthm,amssymb}
\usepackage{enumitem}
\usepackage{color}

\newtheorem{thm}{Theorem}
\newtheorem{lem}[thm]{Lemma}
\newtheorem{conj}[thm]{Conjecture}

\theoremstyle{definition}

\newtheorem{rem}[thm]{Remark}

\newcommand{\PL}[1]{\ensuremath{\mathrm{PL}_{\boldsymbol{#1}}}}
\newcommand{\bu}{\ensuremath{\boldsymbol{u}}}
\newcommand{\N}{\mathbb{N}}
\newcommand{\R}{\mathbb{R}}

\setlist[enumerate]{label=\roman*),itemsep=0pt}

\begin{document}







\title{A note on palindromic length of Sturmian sequences}
\author{Petr Ambro\v{z}, Edita Pelantov\'a\\[1mm]
Department of Mathematics FNSPE\\ Czech Technical University in Prague\\
Trojanova 13, 120 00 Praha 2, Czech Republic}
\date{}

\maketitle

\begin{abstract}
Frid, Puzynina and Zamboni (2013) defined the palindromic length of a finite word $w$ as the minimal number 
of palindromes whose concatenation is equal to $w$. For an infinite word $\bu$ we study $\PL{u}$, that is, 
the function that assigns to each positive integer $n$, the maximal palindromic length of factors of length 
$n$ in $\bu$. Recently, Frid (2018) proved that $\limsup_{n\to\infty}\PL{u}(n)=+\infty$ for any Sturmian word $\bu$.
We show that there is a constant $K>0$ such that $\PL{u}(n)\leq K\ln n$ for every Sturmian word $\bu$, and
that for each non-decreasing function $f$ with property $\lim_{n\to\infty}f(n)=+\infty$ there is a Sturmian
word $\bu$ such that $\PL{u}(n)=\mathcal{O}(f(n))$.
\end{abstract}



\section{Introduction}

Palindromic length of a word $v$, denoted by $|v|_{\text{pal}}$, is the minimal number $K$ of
palindromes $p_1,p_2,\ldots,p_K$ such that $v=p_1p_2\cdots p_K$. This notion has been introduced 
by Frid, Puzynina and Zamboni~\cite{frid-puzynina-zamboni-aam-50} along with the following conjecture.

\begin{conj}
If there is a positive integer $P$ such that $|v|_{\text{pal}}\leq P$ for every factor $v$ of an infinite
word $\boldsymbol{w}$ then $\boldsymbol{w}$ is eventually periodic.
\end{conj}

Frid et al.\ proved validity of the conjecture for $r$-power-free infinite words, i.e., for words which
do not contain factors of the form $v^r=vv\cdots v$ ($r$ times for some integer $r\geq 2$). By result
of Mignosi~\cite{mignosi-tcs-82} the conjecture thus holds for any Sturmian word whose slope has bounded 
coefficients in its continued fraction. Recently, Frid~\cite{frid-ejc-71} proved the conjecture for all Sturmian 
words.

In this paper we study asymptotic growth of function $\PL{u}:\mathbb{N}\rightarrow\mathbb{N}$ defined for an infinite
word $\bu$ by
\[
\PL{u}(n) = \max\{|v|_{\text{pal}} : \text{$v$ is factor of length $n$ in $\boldsymbol{u}$}\}. 
\]
The aforementioned result by Frid can be stated, using function $\PL{u}$, in the form of the following theorem.

\begin{thm}[\cite{frid-ejc-71}]
  Let $\boldsymbol{u}$ be a Sturmian word. Then $\limsup\limits_{n\to\infty}\PL{u}(n)=+\infty$.
\end{thm}

We prove the following two theorems about the rate of growth of function $\PL{u}$
for Sturmian words. 

\begin{thm}\label{thm:libovolne_pomaly_rust}
Let $f:\mathbb{N}\rightarrow\mathbb{R}$ be a non-decreasing function with $\lim\limits_{n\to\infty}f(n)=+\infty$.
Then there is a Sturmian word $\bu$ such that $\PL{u}(n)=o(f(n))$.
\end{thm}

\begin{thm}\label{thm:rust_pomalejsi_nez_ln}
There is a constant $K$ such that for every Sturmian word $\bu$ we have $\PL{u}\leq K\ln n$.
\end{thm}

In other words, $\PL{u}$ may grow into infinity arbitrarily slow (Theorem~\ref{thm:libovolne_pomaly_rust}) 
and not faster than $\mathcal{O}(\ln n)$ (Theorem~\ref{thm:rust_pomalejsi_nez_ln}). Let us stress that
the constant $K$ in Theorem~\ref{thm:rust_pomalejsi_nez_ln} is universal for every Sturmian word.

Both theorems refer to upper
estimates on the growth of $\PL{u}$. Indeed, it is much more difficult to obtain a lower bound on the
growth, such bound is not known even for the Fibonacci word. Recently, Frid~\cite{frid-numeration-2018} considered 
a certain sequence of prefixes of the Fibonacci word, denoted $(w^{(n)})$, and she formulated 
a conjecture about the precise value of $|w^{(n)}|_{\text{pal}}$. This conjecture can be rephrased in the following
way (cf.~Remark~\ref{rem:frid_lower_bound_conjecture}).

\begin{conj}\label{conj:frid_lower_bound_conjecture}
Let $\boldsymbol{f}$ be the Fibonacci word, that is, the fixed point of the morphism $0\mapsto 01$, $1\mapsto 0$. 
Then
\[
\limsup_{n\to\infty}\frac{\PL{f}(n)}{\ln n} \geq \frac{1}{3\ln\tau},
\] 
where $\tau$ is the golden ratio.
\end{conj}

We propose (see Remark~\ref{rem:frid_lower_bound_conjecture} for more details) the following extension of this 
so far unproved statement.

\begin{conj}
Let $\boldsymbol{u}$ be a Sturmian word whose slope has bounded coefficients in its continued fraction.
Then 
\[
\limsup_{n\to\infty}\frac{\PL{u}(n)}{\ln n} > 0.
\]
\end{conj}

\section{Preliminaries}


An \emph{alphabet} $A$ is a finite set of \emph{letters}. A finite sequence of letters of $A$ is called 
a (finite) \emph{word}. The \emph{length} of a word $w=w_1w_2\cdots w_n$, that is, the number of its 
letters, is denoted $|w|=n$. The notation $|w|_a$ is used for the number of occurrences of the letter $a$ in $w$.
The \emph{empty word} is the unique word of length 0, denoted by $\varepsilon$. The set of  all finite words over $A$ 
(including the empty word) is denoted by $A^*$, equipped with the operation of concatenation of words $A^*$ is
a free monoid with $\varepsilon$ as its neutral element.
We consider also \emph{infinite words} $\bu = u_0u_1u_2\cdots$, the set of infinite words over $A$ is 
denoted by $A^{\N}$.

A word $w$ is called a \emph{factor} of $v\in A^*$ if there exist words $w^{(1)},w^{(2)}\in A^*$ such that 
$v = w^{(1)}ww^{(2)}$. The word $w$ is called a \emph{prefix} of $v$ if $w^{(1)}=\varepsilon$, it is called a 
\emph{suffix} of $v$ if $w^{(2)}=\varepsilon$. The notions of factor and prefix can be easily extended to 
infinite words. The set of all factors of an infinite word $\bu$, called the \emph{language} of $\bu$, 
is denoted by $\mathcal{L}(\bu)$. Let $w$ be a prefix of $v$, that is, $v=wu$ for some word $u$. 
Then we write $w^{-1}v=u$.

The \emph{slope} of a nonempty word $w\in\{0,1\}^*$ is the number $\pi(w)=\frac{|w|_1}{|w|}$.
Let $\bu=(u_n)_{n\geq 0}$ be an infinite word. Then the limit
\begin{equation}\label{eq:def_slope}
\rho = \lim_{n\to\infty}\pi(u_0\cdots u_{n-1}) = \frac{|u_0\cdots u_{n-1}|_1}{n}
\end{equation}
is the \emph{slope} of the infinite word. Obviously, the slope of $\bu$ is equal to the frequency of the 
letter 1 in $\bu$.

In this paper we are concerned with the so-called \emph{Sturmian words}~\cite{morse-hedlund-ajm-62}.
These are infinite words over 
a binary alphabet that have exactly $n+1$ factors of length $n$ for each $n\geq 0$. Sturmian words admit 
several equivalent definitions and have many interesting properties.
We will need the following two fact above all. The limit in~(\ref{eq:def_slope}) exists, and thus the
slope of a Sturmian word is well defined, and, moreover, it is an irrational number~\cite{lothaire2}. 
Two Sturmian words have the same language if and only they have the same slope~\cite{mignosi-tcs-65}.


A \emph{morphism} of the free monoid $A^*$ is a map $\varphi:A^*\rightarrow A^*$ such that 
$\varphi(vw)=\varphi(v)\varphi(w)$ for all $v,w\in A^*$.
A morphism $\varphi$ is called \emph{Sturmian} if $\varphi(\bu)$ is a Sturmian word for every
Sturmian word $\bu$. The set of all Sturmian morphisms coincides with the so-called 
\emph{monoid of Sturm}~\cite{mignosi-seebold-jtnb-5}, it is the monoid generated by the following
three morphisms
\[
E: \begin{aligned}0 &\mapsto 1 \\ 1&\mapsto 0\end{aligned}\,,\qquad
G: \begin{aligned}0 &\mapsto 0 \\ 1&\mapsto 01\end{aligned}\,,\qquad
\tilde{G}: \begin{aligned}0 &\mapsto 0 \\ 1&\mapsto 10\end{aligned}\,.
\]

\section{Images of Sturmian words}

In this section we study length and palindromic length of images of words under morphisms
$\psi_b:\{0,1\}^*\rightarrow\{0,1\}^*$, where $b\in\N$, $b\geq 1$ and
\begin{equation}\label{eq:morphism_psi}
\begin{split} \psi_b(0) & = 10^{b-1},  \\ \psi_b(1) &= 10^b. \end{split}
\end{equation}
Note that $\psi_b$ is a Sturmian morphism since $\psi_b = \tilde{G}^{b-1}\circ E\circ G$.

\begin{lem}\label{lem:len_of_image}
Let $b,c\in\N$, $b,c\geq 1$ and let $v\in\{0,1\}^*$. Then
\begin{enumerate}
\item
  $|\psi_b(v)|\geq b|v|$,
\item
  $|(\psi_c\circ\psi_b)(v)|\geq 2|v|$.
\end{enumerate}
\end{lem}
\begin{proof}
i) Let $x=|v|_0$ and $y=|v|_1$. Then $\psi_b(v)$ contains $x':=(b-1)x+by$ zeros and $y':=x+y$ ones.
Thus $|\psi_b(v)| = x'+y' = bx+(b+1)y \geq b(x+y) = b|v|$.

ii) The word $(\psi_c\circ\psi_b)(v)$ contains $x'':=(c-1)x'+cy'$ zeros and $y'':=x'+y'$ ones. Thus
$|(\psi_c\circ\psi_b)(v)|=x''+y''=cx'+(c+1)y'\geq x'+2y'\geq 2y'=2(x+y)=2|v|$.
\end{proof}

\begin{lem}\label{lem:palllen_of_image}
Let $b\in\N$, $b\geq 1$ and let $v\in\{0,1\}^*$, Then $|\psi_b(v)|_{\text{pal}}\leq|v|_{\text{pal}}+1$.
\end{lem}
\begin{proof}
One can easily check that if $p$ is a palindrome then both $\psi_b(p)1$ and $1^{-1}\psi_b(p)$ are
palindromes.

If $v=p_1p_2\cdots p_{2q}$, where all $p_i$ are palindromes, then
\[
\psi_b(v) = 
\underbrace{\psi_b(p_1)1}_{p_1'}\cdot
\underbrace{1^{-1}\psi_b(p_2)}_{p_2'}\cdot
\underbrace{\psi_b(p_3)1}_{p_3'}\cdot
\underbrace{1^{-1}\psi_b(p_4)}_{p_4'}\cdots
\underbrace{\psi_b(p_{2q-1})1}_{p_{2q-1}'}\cdot
\underbrace{1^{-1}\psi_b(p_{2q})}_{p_{2q}'}
\]
is a factorization of $\psi_b(v)$ into $2q$ palindromes and therefore we have 
$|\psi_b(v)|_{\text{pal}}\leq|v|_{\text{pal}}$.

On the other hand, if $|v|_{\text{pal}}$ is odd the factorization of $\psi_b(v)$ is almost the same with the only
exception that at the end there is (possibly non-palindromic) image of the last palindrome, i.e., $\psi_b(p_{2q+1})$.
The statement follows from the fact that $\psi_b(p_{2q+1}) = 1\cdot1^{-1}\psi_b(v_{2q+1})$.
\end{proof}

\begin{lem}\label{lem:cfe_of_image}
Let \bu\ be a Sturmian word with slope $\alpha\in(0,1)$ and let $\alpha=[0,a_1,a_2,a_3,\ldots]$ be
its continued fraction. Then $\psi_b(\bu)$ is a Sturmian word with slope $\beta$,
where $\beta=[0,b,a_1,a_2,a_3,\ldots]$.
\end{lem}
\begin{proof}
Recall that $\alpha$ is the frequency of the letter 1 in \bu, that is,
\[
\alpha = \lim_{|v|\to\infty}\frac{|v|_1}{|v|_0+|v|_1}, \quad \text{where $v\in\mathcal{L}(\bu)$}.
\]
Let us consider the image of $v\in\mathcal{L}(\bu)$ under $\psi_b$. We have
$|\psi_b(v)|_0 = (b-1)|v|_0 + b|v|_1$ and $|\psi_b(v)|_1 = |v|_0 + |v|_1$.
Therefore
\begin{align*}
\beta & = \lim_{|v|\to\infty}\frac{|\psi_b(v)|_1}{|\psi_b(v)|_0+|\psi_b(v)|_1} =
\lim_{|v|\to\infty}\frac{|v|_0 + |v|_1}{b|v|_0 + (b+1)|v|_1} = \\[2mm]
&= \lim_{|v|\to\infty}\frac{1}{b + \frac{|v|_1}{|v|_0+|v|_1}} = \frac{1}{b+\alpha}.
\qedhere
\end{align*}
\end{proof}

\begin{lem}\label{lem:preimage_and_its_pallen}
Let $v\in\{0,1\}^*$ be a factor of a Sturmian word $\bu$ with slope $\beta=[0,b,a_1,a_2,a_3,\ldots]$
and let $|v|_1\geq 2$. Then there are words $v',v_L,v_R$ such that $v'\neq\varepsilon$ is a factor 
of a Sturmian words with slope $\alpha=[0,a_1,a_2,a_3,\ldots]$, $v_L$ is a proper suffix of $\psi_k(x)$ and
$v_R$ is a proper prefix of $\psi_k(y)$ for some $x,y\in\{0,1\}$, and
\begin{enumerate}
\item
  $v = v_L\psi_b(v')v_R$,
\item
  $|v|_{\text{pal}} \leq 4 + |v'|_{\text{pal}}$.
\end{enumerate}
\end{lem}
\begin{proof}
i) Let $\bu$ be a Sturmian word with slope $\alpha=[0,a_1,a_2,a_3,\ldots]$. By Lemma~\ref{lem:cfe_of_image},
$\psi_k(\bu)$ has slope $\beta=[0,b,a_1,a_2,a_3,\ldots]$. Recall that the language of a Sturmian word is 
entirely determined by its slope, thus we have $v\in\mathcal{L}(\psi_b(\bu))$. Since by assumption
$v$ contains at least two ones, we can unambiguously write it in the required form.

ii) This statement then follows from inequalities 
$|v|_{\text{pal}} \leq |v_L|_{\text{pal}} + |\psi_b(v')|_{\text{pal}} + |v_R|_{\text{pal}}$, $|v_L|_{\text{pal}}\leq 1$,
$|v_R|_{\text{pal}}\leq 2$ and from Lemma~\ref{lem:palllen_of_image}.
\end{proof}

\section{Proofs of main Theorems}

Both proofs make use of the following idea. Let $\bu$ be a Sturmian word with slope 
$\alpha=[0,a_1,a_2,a_3,\ldots]$. Let $v=v^{(1)}\in\mathcal{L}(\bu)$. By successive application of 
Lemma~\ref{lem:preimage_and_its_pallen} we find words $v^{(2)}, v^{(3)}, \ldots, v^{(j+1)}$ such that for
every $i=1,2,\ldots,j$ we have
\begin{enumerate}
\item
  $v^{(i)}$ is a factor of a Sturmian word with slope $[0,a_i,a_{i+1},a_{i+2},\ldots]$,
\item
  $|v^{(i)}|\geq |\psi_{a_i}(v^{(i+1)})|\geq a_i|v^{(i+1)}|$ \ 
  (this follows from Lemmas~\ref{lem:preimage_and_its_pallen} and~\ref{lem:len_of_image}),
\item
  $|v^{(i)}|_{\text{pal}} \leq 4 + |v^{(i+1)}|_{\text{pal}}$,
\item
  $v^{(j+1)}$ does not contain two ones, in particular $|v^{(j+1)}|_{\text{pal}}\leq 2$ and
  $|v^{(j+1)}|\geq 1$.
\end{enumerate}
Altogether we have
\begin{equation}\label{eq:pallen_from_iterations}
\begin{split}
|v| = |v^{(1)}| &\geq a_1a_2\cdots a_j, \\
|v|_{\text{pal}} &\leq 4j + 2.
\end{split}
\end{equation}

\begin{proof}[Proof of Theorem~\ref{thm:libovolne_pomaly_rust}]
Let $f:\N\rightarrow\R$ be a non-decreasing function with $\lim_{n\to\infty}f(n)=+\infty$. We find 
$a_1\in\N$, $a_1\geq 2$ such that $f(a_1)\geq 1$, then $a_2\in\N$, $a_2\geq 2$ such that $f(a_1a_2)\geq 2^2$, and
so on, i.e., we proceed recurrently to find $a_k\in\N$, $a_k\geq 2$ such that
\begin{equation}\label{eq:f(a_1...a_k)}
f(a_1a_2\cdots a_k)\geq k^2\quad \text{ for all $k\in\N$, $k\geq 1$}.
\end{equation}
Using~(\ref{eq:pallen_from_iterations}),~(\ref{eq:f(a_1...a_k)}) and monotony of $f$ we can estimate
\[
\frac{|v|_{\text{pal}}}{f(|v|)} \leq \frac{4j+2}{f(a_1a_2\cdots a_j)} \leq \frac{4j+2}{j^2}.
\]
Obviously $j\to\infty$ as $|v|=n\to\infty$ and therefore
\[
\limsup_{n\to\infty} \frac{\PL{u}(n)}{f(n)} \leq \lim_{j\to\infty}\frac{4j+2}{j^2} = 0.
\qedhere
\]
\end{proof}

\begin{proof}[Proof of Theorem~\ref{thm:rust_pomalejsi_nez_ln}]
The estimate $|v|\geq a_1a_2\cdots a_j$ is weak in the case where most of the coefficients of the 
continued fraction are equal to 1. Therefore, we use the fact that $v^{(i)}$ contains factor
$(\psi_{a_i}\circ\psi_{a_{i+1}})(v^{(i+2)})$. By Lemma~\ref{lem:len_of_image} we have $|v^{(i)}|\geq 2|v^{(i+2)}|$
and thus $|v|\geq 2^{\lfloor\frac{j}{2}\rfloor}$. Using this estimate we get
\[
\frac{|v|_{\text{pal}}}{\ln |v|} \leq \frac{4j+2}{\frac{j-1}{2}\ln 2} \xrightarrow{\ j\to\infty\ }\frac{8}{\ln 2}.
\]
Statement of the theorem follows, using $K=\frac{8}{\ln2}$.
\end{proof}

\begin{rem}\label{rem:frid_lower_bound_conjecture}
In~\cite{frid-numeration-2018}, Frid defined the sequence $(w^{(n)})$ of prefixes of the Fibonacci word 
$\boldsymbol{f}$,  where $|w^{(n)}|$ has representation $(100)^{2n-1}101$ in the Ostrowski numeration system.

Using the Fibonacci sequence $(F_n)_{n\geq 0}$ (given by $F_0=1$, $F_2=2$ and $F_{n+2}=F_{n+1}+F_n$ for $n\in\N$)
one gets $|w^{(n)}|=F_0+F_2+\sum_{k=1}^{2n-1}F_{3k+2} < F_{6n}$. Frid proved that $|w^{(n)}|_{\text{pal}}\leq 2n+1$,
while she conjectured that the equality $|w^{(n)}|_{\text{pal}}= 2n+1$ holds. Since 
$F_{n} = \tfrac{1}{\sqrt{5}}\tau^{n+2}(1+o(1))$, the validity of Frid's conjecture would imply
\begin{equation}\label{eq:frid_conjecture_rephrased}
\frac{|w^{(n)}|_{\text{pal}}}{\ln|w^{(n)}|} \geq \frac{2n+1}{\ln F_{6n}} = \frac{2n+1}{(6n+2)\ln\tau(1+o(1))}
\xrightarrow{\;n\to\infty\;}\frac{1}{3\ln\tau}
\end{equation}
as stated in Conjecture~\ref{conj:frid_lower_bound_conjecture}.

In her proof of the fact that for a Sturmian word $\bu$ the function $\PL{u}(n)$ is not bounded,
Frid considered only prefixes of $\bu$. This was made possible by the following result by 
Saarela~\cite{saarela-words-2017}: for a factor $x$ of a word $y$ we have $|x|_{\text{pal}}\leq 2|y|_{\text{pal}}$.
Computer experiments do indicate that the prefixes $w^{(n)}$ have the highest possible ratio
$\frac{|w|_{\text{pal}}}{\ln|w|}$ (among all prefixes of $\boldsymbol{f}$). However, it is still possible that 
there is a sequence of factors 
(not prefixes) of $\boldsymbol{f}$ which can be used to enlarge the constant $\frac{1}{3\ln\tau}$ 
in~(\ref{eq:frid_conjecture_rephrased}).
\end{rem}

\section*{Acknowledgements}

This work was supported by the project CZ.02.1.01/0.0/0.0/16\_019/0000778
from European Regional Development Fund. We also acknowledge financial support of the Grant Agency of the 
Czech Technical University in Prague, grant No.\ SGS14/205/OHK4/3T/14.


\begin{thebibliography}{1}

\expandafter\ifx\csname url\endcsname\relax
   \def\url#1{\texttt{#1}}\fi
\expandafter\ifx\csname urlprefix\endcsname\relax\def\urlprefix{URL }\fi
 \expandafter\ifx\csname href\endcsname\relax
   \def\href#1#2{#2} \def\path#1{#1}\fi

\bibitem{frid-puzynina-zamboni-aam-50}
A.~Frid, S.~Puzynina, L.~Zamboni, On palindromic factorization of words,
  Advances in Applied Mathematics 50~(5) (2013) 737--748.
\newblock \href {https://doi.org/10.1016/j.aam.2013.01.002}
  {\path{doi:10.1016/j.aam.2013.01.002}}.

\bibitem{mignosi-tcs-82}
F.~Mignosi, On the number of factors of {S}turmian words, Theoret. Comput. Sci.
  82~(1) (1991) 71--84.
\newblock \href {https://doi.org/10.1016/0304-3975(91)90172-X}
  {\path{doi:10.1016/0304-3975(91)90172-X}}.

\bibitem{frid-ejc-71}
A.~Frid, Sturmian numeration systems and decompositions to palindromes,
  European Journal of Combinatorics 71 (2018) 202--212.
\newblock \href {https://doi.org/10.1016/j.ejc.2018.04.003}
  {\path{doi:10.1016/j.ejc.2018.04.003}}.

\bibitem{frid-numeration-2018}
A.~Frid, Representations of palindromes in the {F}ibonacci word, in: Numeration
  2018, 2018, pp. 9--12.

\bibitem{morse-hedlund-ajm-62}
M.~Morse, G.~A. Hedlund, Symbolic dynamics {II}. {S}turmian trajectories, Amer.
  J. Math. 62 (1940) 1--42.

\bibitem{lothaire2}
M.~Lothaire, Algebraic {C}ombinatorics on {W}ords, Vol.~90 of Encyclopedia of
  Mathematics and its Applications, Cambridge University Press, Cambridge,
  2002.

\bibitem{mignosi-tcs-65}
F.~Mignosi, Infinite words with linear subword complexity, Theoret. Comput.
  Sci. 65~(2) (1989) 221--242.
\newblock \href {https://doi.org/10.1016/0304-3975(89)90046-7}
  {\path{doi:10.1016/0304-3975(89)90046-7}}.

\bibitem{mignosi-seebold-jtnb-5}
F.~Mignosi, P.~Séébold, Morphismes sturmiens et r{\`e}gles de {R}auzy,
  J.~Th\'eor. Nombres Bordeaux 5~(2) (1993) 221--233.

\bibitem{saarela-words-2017}
A.~Saarela, Palindromic length in free monoids and free groups, in:
  Combinatorics on words, Vol. 10432 of Lecture Notes in Comput. Sci.,
  Springer, Cham, 2017, pp. 203--213.

\end{thebibliography}
\end{document}